\documentclass{amsart}
\usepackage[utf8]{inputenc}
\usepackage{amssymb, tipa, thmtools}
\usepackage{graphicx}
\usepackage{xcolor}
\usepackage{tikz}
\usepackage{csquotes}
\usepackage[english]{babel}
\usepackage[fixlanguage, datename]{babelbib}
\usepackage{hyperref}

\newtheorem{theorem}{Theorem}[section]

\newtheorem{lemma}[theorem]{Lemma}
\newtheorem{remark}[theorem]{Remark}
\newtheorem{corollary}[theorem]{Corollary}
\newtheorem{definition}[theorem]{Definition}
\newtheorem{example}[theorem]{Example}

\newtheorem{fact}[theorem]{Fact}
\newtheorem{proposition}[theorem]{Proposition}

\usepackage{mathtools}

\def\R{\mathbb{R}}

\def\Z{\mathbb{Z}}

\def\lang{\mathcal{L}}
\def\graph{\mathcal{G}}

\newcommand{\set}[1]{\{#1 \}}
\newcommand{\inv}{^{-1}}

\newcommand{\model}{\mathcal{M}}
\newcommand{\nodel}{\mathcal{N}}

\DeclareMathOperator{\acl}{acl}
\DeclareMathOperator{\Aut}{Aut}

\title{Transitive Extensions of Automorphism Groups of Generic Structures}
\author{Felipe Estrada}
\thanks{I would like to thank Professor Dugald Macpherson for proposing the topic of this paper and for offering extensive and constructive feedback at every stage, which significantly strengthened this manuscript}
\thanks{The author was supported by grant INV-2025-216-3443 from the Faculty of Sciences of the Universidad de los Andes while carrying out this research.}

\begin{document}

\maketitle

\begin{abstract}
    This work addresses the existence of transitive extensions of certain infinite permutation groups which arise as the automorphism groups of model-theoretic structures which are generic in the Fraïssé sense. The study of transitive extensions has hitherto largely concerned itself with finite permutation groups. Moving beyond the finite realm, we develop combinatorial tools to prove that transitive extensions exist for edge-colored $k$-hypergraphs only when the number of colors is a power of two and that transitive extensions exist for $k$-hypertournaments (in the Cherlin sense) only when $k$ is even, among other results.
\end{abstract}

\section{Introduction}

Let $\model$ be a first-order structure with universe $M$. A \emph{one-point extension} of $\model$ is a first-order structure $\model'$ on $M\cup\set{0}$ where $0$ is a new element such that $(\Aut(\model')_0, M) = (\Aut(\model), M)$. Here $\Aut(\model')_0$ is the stabilizer of 0 in $\Aut(\model')$. By a \emph{transitive extension}, we mean a 1-point extension whose automorphism group is 1-transitive. This is not to be confused with an \emph{expansion} of $\model$, which consists of adding new symbols to the language and assigning an interpretation of them in $M$.

Transitive extensions have been studied extensively for finite permutation groups (see e.g. \cite{biggs1979permutation}), but there has been little work on transitive extensions of infinite structures.  The existence of transitive extensions of the Rado graph, the random tournament, and homogeneous C-sets were already known to some experts \cite{CameronPersonalComm}, but their construction was never formally written down. In this article we will construct transitive extensions of classes of structures which include these examples and more.

A closely related concept to transitive extensions is that of reducts. Given a structure $\model$, let $\nodel$ be another first-order structure with the same universe $M$ (but generally a different language). There are two senses in which $\nodel$ may be a reduct of $\model$. We shall call $\nodel$ a \emph{general reduct} if $\Aut(\model) \leq \Aut(\nodel)$, and a \emph{definable reduct} if there exist a set $\lang'$ of 0-definable relations and functions on $\model$ such that $\model|_{\lang'}$ is isomorphic to $\nodel$. Clearly, any definable reduct is a general reduct. By the Ryll-Nardzewski Theorem the converse holds if $\model$ is an $\omega$-categorical structure, which covers most examples in this work. We shall mostly explore general reducts, and refer to them simply as reducts. 

The exact relationship between transitive extensions and reducts will be elaborated upon in Lemma \ref{reducts lemma}, but suffice to say that with some mild conditions, a transitive extension is isomorphic to a reduct of the original structure. Since many structures have had their reducts classified, for these structures there is a nice shortlist of possibilities for what a transitive extension of their automorphism group may be. Certain structures are known to have no proper non-trivial reducts, and that fact may be used to quickly conclude no transitive extensions exist.

The general goal of this work is to examine a series of generic structures and characterize the conditions under which they have transitive extensions. Following this introduction, Section 2 will begin by establishing a number of preliminary general results about transitive extensions which apply to all structures in subsequent sections, as well as formally defining our notion of genericity.

Section 3 defines edge-colored $k$-hypergraphs and develops the first main result:

\begin{restatable}{maintheorem}{maintheoremone}\label{main graph theorem}
    The generic edge-colored $k$-hypergraph with $n$ colors has a transitive extension if and only if $n$ is a power of two.
\end{restatable}

The section additionally provides an explicit construction of these transitive extensions when they exist: they are $(k+1)$-hypergraphs which satisfy a parity condition. For $k=2$ this is the only case of a nontrival (i.e. not interdefinable with a pure set) ultrahomogeneous graph with a transitive extension.

Section 4 is split into two parts, each of which deals with a distinct generalization of the random tournament to higher dimensions. Section 4.1 will deal with Cherlin's notion of a $k$-hypertournament, which we rename as a $k$-orientation to avoid ambiguity. We prove our second main result:

\begin{restatable}{maintheorem}{maintheoremtwo}\label{main orient theorem}
    The generic $k$-orientation has a transitive extension if and only if $k$ is even.
\end{restatable}

Similarly to the first main theorem, we provide explicit constructions when they exist, and realize them as $(k+1)$-orientations with a parity condition, and conclude the section with a comment on how these results may be leveraged to obtain transitive extensions for random digraphs. Section 4.2 deals with the more standard definition of $k$-hypertournament. For these, we leverage the results of Section 3 to prove transitive extensions do not exist except in the $k = 2$ case, which is simply the random tournament. In section 5 we develop another purely negative result by proving that no non-trivial regular equivalence relation has a transitive extension. 

Lastly, we conclude this work in section 6, in which we slightly extend the fact, implicit in Adeleke and Neumann's work \cite{adeleke1998relations}, that certain $C$-sets, which are ternary relational structures which capture the structure of rooted trees, have transitive extensions which are $D$-sets, related structures which capture the structure of unrooted trees. The main focus for us is not $C$-sets themselves but rather three natural expansions of their theory: ordered, internally-colored, and leveled $C$-sets. We cap off this work by constructing transitive extensions for the first two expansions and prove their nonexistence for the third.

\section{General Results}

We begin with this elementary result which is crucial to the study of transitive extensions.

\begin{fact}\label{transitivity degree jump fact}
    If $\model$ is $k$-transitive and has a transitive extension $\mathcal{M}'$, then $\mathcal{M}'$ is $(k+1)$-transitive.
\end{fact}
A proof of this may be found in \cite{biggs1979permutation}. It is stated for finite permutation groups only but the proof applies equally in the infinite case.

We shall now elaborate on the precise relationship between reducts and transitive extensions. We call a reduct of a structure proper if it is not isomorphic to the original structure, and nontrivial if it is not isomorphic to a pure set.

\begin{example}
    Let $\model$ be an $\lang$-structure with universe $M$. Define its trivial extension $\model_0$ as a $(\lang \cup \set{0})$-structure with universe $M \cup \set{0}$, with $0$ a new constant symbol interpreted as itself, with language symbols in $\lang$ interpreted such that relations and functions on $\model_0$ coincide with $\model$ on $M$, relations are automatically false on tuples involving $0$, and functions on tuples involving $0$ always yield $0$. 
\end{example}

It is simple to check that the trivial extension is indeed a one-point extension, albeit clearly not a transitive one.

\begin{remark}
    Let $\model'$ be any one-point extension of $\model$. The Stabilizer $\Aut(\model')_0$ acts on $M$ identically to the automorphism group of the trivial extension of $\model$.
\end{remark}

That is to say, any extension, in particular a transitive extension, is always a (general) reduct of the trivial extension. Since the trivial extension has an automorphism group that is essentially the same as the original structure's, in an informal sense this tells us any transitive extension is a sort of reduct of the original. This notion can be made formal with a few additional assumptions:

\begin{lemma}\label{reducts lemma}
    Suppose $\mathcal{M}$ has a transitive one-point extension $\mathcal{M}'$ and there is an $a_0 \in M$ and an isomorphism $\phi$ mapping $M$ to $M \setminus \set{a_0}$ as a substructure of $\model$. Then $\mathcal{M}'$ is isomorphic to a reduct of $\mathcal{M}$.
\end{lemma}
\begin{proof}
    Map the elements of $M'$ to $M$ by $0 \mapsto a_0, a \mapsto\phi(a)$ for $a \in A$, and let $\mathcal{M}''$ be the induced structure on $M$ which makes this map an isomorphism. This produces a copy of $\mathcal{\model'}$ in $M$ which is preserved by the action of $\Aut(\model')_0$, which by hypothesis is equal to $\Aut(\model)$. Thus $\model$ contains a structure whose automorphism group extends $\Aut(M)$, which is therefore a reduct.
\end{proof}

\begin{corollary}\label{reducts corollary}
    Let $\model$ be an ultrahomogeneous structure with trivial algebraic closure. If $\model$ has no proper, non-trivial reducts, then $\model$ does not have a transitive extension, unless $\model$ is interdefinable with a pure set.
\end{corollary}
\begin{proof}
    If $\model$ is ultrahomogeneous and has trivial algebraic closure, a map $\phi$ satisfying the conditions of the previous lemma may be constructed by back-and-forth for an arbitrary $a_0 \in M$, and so if a transitive one-point extension $\model'$ exists it would be a reduct of the original. However, the degree of transitivity of $\model'$ is always one higher than that of $\model$, so unless $\model$ is $k$-transitive for all $k$, $\model'$ would have to be a nontrivial proper reduct. If $\model$ is in fact $k$-transitive for all $k$ then it is interdefinable with a pure set.
\end{proof}

In many cases, reducts of structures have been classified, and as such the existence of their transitive extensions may be checked by exhaustively examining each of their reducts. For instance, Agarwal and Kompatscher \cite{agarwal2018pairwise} classified the reducts of the Henson digraphs, and obtain $2^{\aleph_0}$ pairswise non-isomorphic Henson digraphs (including all Henson graphs) without any proper, nontrivial reducts, and hence without transitive extensions. In previous work Agarwal \cite{agarwal2016reducts} classified the reducts of the generic digraph, which is not a focus of study of this work, but a brief note is made on constructing its transitive extension at the end of section 4.2. Other such classification results on transitive structures include \cite{bodirsky201542}, \cite{pach2014reducts}, \cite{thomas1991reducts} and \cite{thomas1996reducts}. In this work, however, we generally take a more direct approach to constructing transitive extensions or proving their nonexistence.

The next remark is simple, but useful when examining expansions of structures which are known to have a transitive extension. 

\begin{remark}\label{interpretation lemma}
    Suppose $\model$ and $\nodel$ are $\omega$-categorical first-order structures with one-point extensions $\model'$ and $\nodel'$, respectively. If $\nodel$ interprets $\model$, then $\nodel'$ interprets $\model'$. 
\end{remark}
\begin{proof}
    Augment the languages of with a new constant 0 for the extra point in each extension. By hypothesis $(\nodel', 0)$ is bi-interpretable with $\nodel$, and similarly $(\model', 0)$ is bi-interpretable with $\model$. Thus, we have a chain of interpretations
    $$\nodel' \rightarrow (\nodel', 0) \rightarrow \nodel \rightarrow \model \rightarrow (\model', 0) \rightarrow \model'$$
    which when composed yield the result. Here $A \to B$ indicates an interpretation of $B$ in $A$. The two outermost interpretations are simply the respective identity maps.
\end{proof}
Of note, this chain of interpretations is constructive. So moreover a specific interpretation of $\model$ in $\nodel$ induces a natural interpretation of $\model'$ in $\nodel'$. In practice we can use this when we know a certain structure has a transitive one-point extension, and we wish to see if an expansion of this structure also has a transitive one-point extension, as such an extension necessarily must expand the previous known one, which sometimes yields a contradiction. (See the section on leveled C-sets for an example.)

The following result is exceedingly important.

\begin{lemma}\label{arity growth lemma}
    Let $k \geq 1$ and suppose $\mathcal{M}$ is a $k$-transitive ultrahomogeneous structure consisting of relations $R_1, \dots, R_n$. If $\model$ has a transitive extension $\model'$, then it has a transitive extension with the same automorphism group as $\model'$ whose language consists of relations $R'_1, \dots, R'_n$, such that each $R'_i$ has arity one higher than the corresponding $R_i$ and the rule
    $$R_i(\overline{x}) \Longleftrightarrow R_i'(\overline{x}, 0)$$
    holds for all tuples $\overline{x}$ from $M$ of the appropriate length.
\end{lemma}
\begin{proof}
    Without loss of generality, we may assume that for each $n$, the relations $R_i$ of arity $n$ form a partition of $M^n$, and that $\model'$ consists of an $n$-ary relation symbol for each orbit of $\Aut(\model')$ on $n$ elements. Given that $\model$ is $k$-transitive, $\model'$ must be at least $(k+1)$-transitive, and so we may assume there are no symbols of arity less than $k+2$ in $\model'$. In fact, since $\model$ is ultrahomogeneous, $k$ is exactly one less than the minimum arity of $R_1, \dots, R_n$.

    Consider now an individual relation $R_i$ of arity $n$. By ultrahomogeneity of $\model$, $R_i$ is precisely one of the orbits on $n$-tuples of $\Aut(\model)$, which are in natural bijection with orbits of $(n+1)$-tuples of $\Aut(\graph')_0$ which end with 0 (and therefore, by transitivity, with all $(n+1)$-orbits in $\Aut(\graph')$). If we introduce a symbol $R_i'$ into $\model'$ for the $(n+1)$-orbit in $\model'$ corresponding to $R_i$, then it satisfies the rule in the statement, and the set of relations of arity $n+1$ partitions $(M')^{n+1}$.

    Let $n_{1}, n_2$ denote the minimum and maximum arity of the relations $R_1, \dots, R_n$, Since the orbits on $m$-tuples in $\Aut(\model)$ for any $m$ are completely determined by the orbits on $n$-tuples for $n_1 \leq n \leq n_2$, the orbits on $(m+1)$-tuples in $\Aut(\model')$ which end with 0 are determined by orbits on $n$-tuples which end with 0 for $n_1 +1 \leq n \leq n_2+1$, and then by transitivity all $(m+1)$-orbits are determined by such $n$-orbits. That is to say any relations in $\model'$ are necessarily deducible from the relations $R_1', \dots, R_n'$, and so can be committed without changing the automorphism group, which completes the proof.
\end{proof}

Hereafter, we will assume without loss of generality that all transitive extensions are of this form. We shall call the relations $R'_i$ the \emph{transitive analogue} of their respective $R_i$, and an extension which satisfies this lemma is said to be in \emph{canonical form}.

\begin{lemma}\label{Extensions of homogeneous structures}
    Suppose a structure $\model$ is ultrahomogeneous in a finite relational language. If $\model$ has a transitive extension $\model'$ in canonical form, then $\model'$ is ultrahomogenous.
\end{lemma}
\begin{proof}
    Suppose $A$ and $B$ are two finitely generated substructures of this expansion of $\model'$ and $\phi$ a partial isomorphism between them. We claim that, without loss of generality, both $A$ and $B$ contain the element 0 and also $\phi(0) = 0$. To see why we may do this, note that if $A$ doesn't contain $0$, we may pick an arbitrary element $a \in A$ and pick an automorphism of $\mathcal{M}$ which maps $a$ to $0$, which exists by transitivity, and work with the image of $A$ under this automorphism instead of with $A$ itself. After that, if $\phi(0) \neq 0$, we may do the same with an automorphism mapping $\phi(0)$ to 0 and replace $B$ with its image under it.

    With this assumption, since $\Aut(\model ')_0 = \Aut (\model)$ and $\phi(0) = 0$, Because we chose $\model'$ as described in Lemma \ref{arity growth lemma}, $\phi\setminus\set{0 \mapsto 0}$ is a partial isomorphism on $\model$ and so extends to a full automorphism $\psi$ of $\model$. Then $\psi \cup\set{0 \mapsto 0} \in \Aut(\model')$ and extends $\phi$, as desired. 
\end{proof}

We shall be mostly exploring the transitive extensions of generic structures, such as $(\mathbb{Q}, <)$, the Rado Graph and the random tournament. The term ``generic'' carries the following formal meaning:

\begin{definition}
    Let $\mathcal{F}$ be a class of finitely-generated first-order structures. If $\mathcal{F}$ is a Fraïssé class, we call its Fraïssé limit the generic member of the class (even though technically not a member of it). That is, the generic member of $\mathcal{F}$ is the unique (up to isomorphism) countable, ultrahomogeneous structure whose age is $\mathcal{F}$, if it exists.
\end{definition}

Occasionally we will want to combine generic structures. For this we will require a slightly stronger condition than being a mere Fraïssé class.

\begin{definition}
    A class of finite structures has the \emph{disjoint Amalgamation Property} (dAP) if for any $A, B, C \in \mathcal{F}$ and embeddings $f_1: A \rightarrow B$, $f_2: A \rightarrow C$ there is a $D \in \mathcal{F}$ and embeddings $g_1: B \rightarrow D$, $g_2: C \rightarrow D$ such that $g_1 \circ f_1 = g_2 \circ g_1$ and $g_1\circ f_1 (A) = g_2 \circ g_2(A) = h_1(B) \cap h_2 (C)$. 
\end{definition}

That is to say the images of $B$ and $C$ in the amalgamation $D$ are ``as disjoint as possible''. The following fact is from \cite[\S 3.3]{Cameron_1990}.

\begin{fact}
    Suppose for $i = 1, 2$, $\mathcal{F}_i$ is a Fraïssé class in a language $\lang_i$, $\lang_1 \cap \lang_2 = \emptyset$. Let the \emph{merge} $\mathcal{F}^*$ of $\mathcal{F}_1$ and $\mathcal{F}_2$ be the $(\lang_1 \cup \lang_2)$-class consisting of all $(\lang_1\cup \lang_2)$-structures $A$ such that $A|_{\lang_i} \in \mathcal{F}_i$.

    If both $\mathcal{F}_1$, $\mathcal{F}_2$ have dAP, then $\mathcal{F}^*$ is a Fraïssé class with dAP.
\end{fact}

If $\mathcal{M}_1$, $\mathcal{M}_2$ are limits of Fraïssé classes with dAP, then we call the limit of their merge the merge of $\mathcal{M}_1$ and $\mathcal{M}_2$

A helpful characterization here is that a Fraïssé class has dAP if and only if its limit $\model$ has trivial algebraic closure, that is $A = \acl(A)$ for all finite $A \subseteq M$ \cite{Cameron_1990}. This makes it easy to see that all generic structures in this paper are Fraïssé limits of classes with dAP (save for one exception which will be pointed out). In practice, we will simply take two ultrahomogeneous countable structures with trivial algebraic closure and freely talk about their merge, when relevant.

\begin{lemma}\label{merge lemma}
    If $\model_3$ is the merge of two ultrahomogeneous, countable structures  $\model_1, \model_2$ with trivial algebraic closure, and $\model_1$, $\model_2$ both have transitive extensions, then $\model_3$ has a transitive extension, which is isomorphic to the merge of the transitive extensions of $\model_1$ and $\model_2$.
\end{lemma}
\begin{proof}
    By Lemma \ref{Extensions of homogeneous structures} the extensions of $\model_1$, $\model_2$ are ultrahomogeneous as well, and it is straightforward to check they too have trivial algebraic closure, so they have a merge $\model_3'$.

    Without loss of generality we may assume the structures $\model_1, \model_2$, $\model_3$ have the same base set, and therefore so do the extensions of $\model_1, \model_2$ and their merge. With this assumption, $\model_3'$ in fact becomes equal to the transitive extension of $\model_3$, which is a simple bookkeeping exercise using the fact that each orbit of the automorphism group on $n$-tuples exactly corresponds to the embeddings of a given finite substructure of size $n$. 
\end{proof}

This last result is trivial, but will be stated since it comes up frequently

\begin{remark}[The Homogeneity Test]\label{homogeneity test}
    If $\mathcal{M}$ is a relational ultrahomogeneous structure and has no 0-definable unary relations beyond the trivial predicates for $\emptyset$ and $M$, then its automorphism group is transitive. More generally, if $\mathcal{M}$ is a relational ultrahomogeneous structure with no non-trivial 0-definable relations of arity $k$ or less, then it is $k$-transitive
\end{remark}
\begin{proof}
    Any two (ordered) substructures of size less than $k$ are isomorphic, hence there is an isomorphism mapping one to the other, and so they are in the same orbit.
\end{proof}

We conclude this section with a simple example of a transitive extension, whose behavior has been observed previously in \cite{huntington1935inter}. This result is not of much interest in itself, but will become relevant in sections \ref{tournaments} and \ref{C sets}.

\begin{example}
    Let $\mathcal{Q} = (Q, <)$ be a generic linear order (that is, a dense linear order without endpoints). Then $\mathcal{Q}$ has a transitive extension which is a dense circular order.
\end{example}
\begin{proof}
    Let $C(x,y,z)$ be the circular order on $Q' = Q \cup \set{0}$ generated by the following rule, which should be familiar in light of Lemma \ref{arity growth lemma}:
    $$C(a, b, 0) \Longleftrightarrow a<b$$
    or equivalently, and more explicitly, for $x,y,z\neq 0$
    $C(x,y,z) \Longleftrightarrow (x<y<z) \vee (y<z<x) \vee (z<y<x)$
    The axioms of a circular order are a straightforward consequence of the linear order axioms of $<$, and the density of $C$ follows from the fact that $(Q, <)$ is dense and unbounded. A permutation $\phi$ of $Q$ respects the linear order if and only if $\phi \cup\set{0 \mapsto 0}$ preserves the circular order, by construction, so indeed $(Q', C)$ is an extension, and moreover is transitive by the homogeneity test, as a dense circular ordering is easily shown to be ultrahomogeneous by back-and-forth.
\end{proof}

\section{Graphs and Hypergraphs}

In this section we shall prove the first major theorem of this work. In the 1980s, it was observed by Peter Cameron (and confirmed in personal communications) that the generic graph, also known as the Rado Graph, has a transitive extension isomorphic to the generic even 3-hypergraph (usually called the two-graph, soon to be defined). Our main theorem of this section will be a generalization of this fact to edge-colored $k$-hypergraphs. We shall treat the uncolored case first as Lemma \ref{Extension of Rado Graph}, and then introduce the edge colorings as well as the main theorem. 

Recall that a $k$-ary relation $R$ is said to be symmetric under a permutation $\pi \in S_k$ if $\forall \overline{x}[R(\overline{x}) \Leftrightarrow R((\pi(\overline{x}))]$, and it is said to be irreflexive if $\lnot R(\overline{x})$ for any $\overline{x}$ whose entries are not all distinct.

\begin{definition}
    Let $k \geq 2$. A (uniform) $k$-hypergraph is a structure $(G, R)$ where $R$ is a $k$-ary relation called the hyperedge relation which is irreflexive and symmetric on all permutations of $k$ elements. A 2-hypergraph is simply called a graph.

    A $k$-hypergraph $G$ is called \emph{even} if every set of $k+1$ points in $G$ has an even number of size $k$ subsets which form a $k$-hyperedge.
\end{definition}

In the literature, a non-uniform $k$-hypergraph may have relations of any arity $n \leq k$. Throughout this work, all $k$-hypergraphs are assumed to be uniform unless otherwise stated. Note that the previous definition can be extended to $k=1$, yielding a colored set. This case has been excluded as the proof of Lemma \ref{Conditions for generic Graph extensions} does not follow for it, and more generally because the study of transitive extensions generally concerns itself with extensions of structures which are already transitive.

\begin{proposition}\label{Extension of Rado Graph}
    Let $\graph = (G, R)$ be the generic $k$-hypergraph. Let $\graph' = (G', R') = (G \cup \set{0}, R')$ be the unique \emph{even} $(k+1)$-hypergraph which satisfies $$ R(\overline{x}) \Longleftrightarrow R'(\overline{x},0)$$ for all $k$-tuples $\overline{x}$. Then $\graph'$ is the generic even $(k+1)$-hypergraph and is a transitive one-point extension of $\graph$.
\end{proposition}
\begin{proof}
    First, we must check that $\graph'$ is well-defined. Let $\overline{a} = (a_0, \dots, a_k) \in G^{k+1}$ be made of distinct elements. Notice the values of $R'(\overline{a}',0)$ for $\overline{a}'$ $k$-subtuple of $\overline{a}$ are determined by $R(\overline{a})$, so, for $\graph'$ to be an even $(k+1)$-hypergraph, the value of $R'(\overline{a})$ must be chosen so that there are an even number of hyperedges on $\overline{a} \cup\set{0}$. In other words, for $\graph'$ to be an even $(k+1)$-hypergraph it is necessary that $R'(\overline{a})$ holds if and only if $\graph$ has an odd number of $k$-hyperedges in $\overline{a}$.

    This fully characterizes the relation $R'$. The question of whether any $(k+2)$-set from $\graph '$ has an even number of $(k+1)$-hyperedges, then, reduces to the claim that in any $(k+2)$-set from $\graph$ (that is, in any $k$-hypergraph with $k+2$ vertices), there are an even number of $(k+1)$-sets of vertices which feature an odd number of $k$-hyperedges between them. This can be seen to be true by induction on the number of $k$-hyperedges, noticing that the empty graph on $k+2$ vertices satisfies the condition, and that adding a $k$-hyperedge to any graph on $k+2$ vertices adds a $k$-hyperedge to exactly two of its $(k+1)$-element subsets; namely, if the graph has vertices $\set{a_1, \dots, a_{k+2}}$ and $E = \set{a_1, \dots, a_{k}}$ is a $k$-hyperedge, it is contained in precisely the $(k+1)$-element subsets $E\cup\set{a_{k+1}}$ $E\cup\set{a_{k+2}}$. Thus by induction, the parity of the number of $(k+1)$-element subsets with an odd number of $k$-hyperedges is always even in a graph with $k+2$ vertices, which yields that $(G', R')$ is a an even $(k+1)$-hypergraph.

    It remains to check that this is a transitive extension. The fact that it is an extension is easy: Let $\phi$ be a permutation on $G$, then $\phi$ is an automorphism of $\graph$ if and only if it preserves the relation $R$, which by construction occurs if and only if $\phi\cup\{0\mapsto 0\}$ preserves $R'$ on $(k+1)$-tuples of the form $(\overline{a}, 0)$, and since whether any $(k+1)$-tuple is a hyperedge is determined by how many $k$-tuples from it form a hyperedge with 0, this is equivalent to it preserving $R'$ on all tuples, which is itself equivalent to being in $\Aut(\graph')_0$. This proves it is an extension.

    Next, we show $\mathcal{G}'$ is ultrahomogeneous, and therefore transitive by the Homogeneity test. First we show that if $H$ is a finite even $(k+1)$-hypergraph, then $H$ embeds into $G'$. To see this, pick an arbitrary $h_0 \in H$ and define a $k$-hypergraph relation on $H \setminus \set{h_0}$ where a $k$-tuple $\overline{h}$ is a $k$-hyperedge if $\overline{h}h_0$ is a $(k+1)$-hyperedge in $H$. Then $H \setminus \set{h_0}$ embeds into $G$ by hypothesis, and now appending $\set{h_0 \mapsto 0}$ to this embedding yields an embedding of $H$ into $G'$. Thus the age of $G'$ includes all finite even $(k+1)$-hypergraphs. 
    
    Now, note that while the previous embedding always embeds into a subset of $\graph'$ which includes 0, we can add an additional condition that the image of the embedding does not include zero, by including an additional ``dummy'' point in $H$ (and any hyperedges necessary to ensure evenness), picking it as $h_0$ and then removing it at the end. Now, if $A$ and $B$ are isomorphic finite substructures of $\graph'$ and $a\not \in A$, we can find isomorphic copies of $A \cup \set{a}$ and $B$ in $\graph'$ which do not include 0. That is, for the sake of back-and-forth we may assume without loss that neither subset contains 0, but then we may view them as subsets of $\graph$, and their (quantifier-free) type in $\graph$ determines their (quantifier-free) type in $\graph'$. By the ultrahomogeneity of $\graph$ we can find a $b$ with the same type over $B$ as the type of $a$ over $A$, and augment the isomorphism between $A$ and $B$ with $a \mapsto b$, yielding a partial isomorphism of $\graph$ and therefore of $\graph'$. By continuing this process in a back-and-forth fashion we may extend to a full automorphism of $\graph'$ and conclude that $\graph'$ is ultrahomogeneous, and therefore it is the generic even $(k+1)$-hypergraph and a transitive extension of $\graph$.
\end{proof}

\subsection{edge-colored (hyper)graphs}\hfill\\

We will now extend this result to graphs in which the edges have been colored. We denote the diagonal of a cartesian power $G^k$ as $\Delta^k_G := \set{(x_1, \dots, x_k)\in G^k: |\set{x_1, \dots, x_k}|<k } $).

\begin{definition}
    An edge-colored $k$-hypergraph (with $n$ colors) is a structure $(G, R_1, \dots, R_n)$ where each $R_i$ is a $k$-hyperedge relation, and $\{R_1, \dots, R_n)$ is a partition of $G^k \setminus \Delta^k_G$.
\end{definition}

Note that we disallow edges with no color, so a usual $k$-hypergraph is seen as an edge-colored $k$-hypergraph with two colors. A $k$-hypergraph with one color is definitionally equivalent to a pure set.

The main theorem we wish to prove in this section is the following:

\maintheoremone*

Another way of stating this result is through the observation that an edge-colored generic $k$-hypergraph with $2^n$ colors is essentially the same as the merge of $n$ uncolored $k$-hypergraphs, so the ``if'' direction falls out easily:

\begin{corollary}[of Proposition \ref{Extension of Rado Graph}]\label{if direction of graph theorem}
    Let $n$ be a natural number and $\graph$ be the generic edge-colored $k$-hypergraph with $2^n$ colors, then $\graph$ has a transitive one-point extension $\graph'$ which is the merge of $n$ generic even $(k+1)$-hypergraphs.
\end{corollary}
\begin{proof}
    If we index the colors of a generic edge-colored $k$-hypergraph with $2^n$ colors with the binary sequences of length $n$, and define $k$-hyperedge $R_1, \dots, R_n$ relations by $R_i(\overline{x})$ if and only if the color of $\overline{x}$ has a 1 in the $i$'th entry, then it is easy to check that $\graph$ is the merge of the generic $k$-hypergraphs $(G, R_i)$, and therefore has the transitive extension described in the statement.
\end{proof}

For the ``only if'' direction, we begin with the following strengthening of Lemma \ref{arity growth lemma}
\begin{lemma}\label{Conditions for generic Graph extensions}
    Let $\graph = (G, R_1, \dots, R_n)$ be an ultrahomogeneous edge-colored $k$-hypergraph in $n$ colors. If $\graph$ has a transitive one-point extension $\graph'$ in canonical form, then $\graph'$ is an edge-colored $(k+1)$-hypergraph $(G, R'_1, \dots, R'_n)$ with $n$ colors and no additional structure.
\end{lemma}
\begin{proof}
    The only thing we need to check is that the relations $R_i'$ are necessarily $(k+1)$-hypergraph relations, that is to say that they are symmetric under all permutations.
    
    Notice that any $(k+1)$-tuple in $\graph'$ is locally isomorphic to any permutation of itself. This is because any permutation of $k$ elements on $\graph$ is a local isomorphism, so any permutation of $k$ elements which additionally fixes 0 is a local isomorphism on $\graph'$, but this property is part of the type of $0$, which has to be the same as the type of any other element, so in fact for any $k+1$ elements, any permutation which fixes one of them is a local isomorphism, such maps generate all permutations on $k+1$ elements (note that this only follows when $k>1)$, so all $R_i'$ are symmetric under all permutations, as desired.
\end{proof}

The question of extending an edge-colored $k$-hypergraph, then, boils down to finding appropriate $(k+1)$-hypergraphs $R_i'$, or proving no such hypergraphs exist. This lemma forces $(k+1)$-tuples of the form $\{\overline{a}, 0\}$ to have the same ``$(k+1)$-color'' as the ``$k$-color'' of $\overline{a}$, but says nothing about $(k+1)$-tuples which do not include 0. Thus, the only real question is whether it is possible to set the colors of $(k+1)$-tuples not including 0 in a way that allows for $\graph'$ to be a transitive extension. We shall show it is not possible when $n$ is not a power of 2. We begin with the following result, which pertains only to the $k = 2$ case.

Henceforth, $M_n^m$ denotes the collection of size $m$ multisets with elements taken from $\set{1, \dots, n}$.

\begin{lemma}\label{existence of palettes}
    Let $\graph$ be the generic edge-colored graph with $n$ colors. If $\graph$ has a transitive one-point extension $\graph'$, then there exists a set $A \subseteq M_n^4$ of 4-multisets such that
    \begin{enumerate}
        \item Every 3-multiset in $M_n^3$ is contained in a unique multiset in $A$.

        \item All multisets of the form $\set{i, i, j, j}$ are in $A$ (Including when $i = j$).

        \item If $\set{i, i', j, j'} \in A$ and $\set{i, i', k, k'} \in A$, then $\set{j, j', k, k'}, \in A$
    \end{enumerate}
\end{lemma}
\begin{proof}
    We shall construct the set $A$ starting from a transitive extension $\graph'$ of $\graph$.

    By Lemma \ref{Conditions for generic Graph extensions}, $(G')^3\setminus \Delta_{G '}^3$ can be colored as $n$ disjoint hypergraphs, $R_1', \dots, R_n'$, which correspond to the edge colors $R_1, \dots, R_n$ of $\graph$ by the rule
    $$R_i(x,y) \Longleftrightarrow R_i'(x,y, 0).$$
    Now, given a triples of elements $(a,b,c)$ in ${G}$, their setwise orbit in $\graph$ is completely determined by the distribution of colors of the triangle they form. i.e. whether $\set{a,b,c}$ maps to $\set{a', b', c'}$ as a set depends only on whether the number of edges of each color is the same for both of them. In $\graph'$ this same setwise orbit is instead determined by which color the hyperedge between them is. Thus, the latter must be determined from the former. That is, there exists a function
    $$f : M_n^3 \rightarrow \set{1, \dots,n}$$
    which maps the multiset of colors on a triangle to the color of the corresponding hyperedge. Another way to think of the input of $f$ is not as the colors of the triangle formed by a triple $\{a,b,c\}$ in $\graph$, but rather as the colors of the hyperedges $\{a, b, 0\}$, $\{a, c, 0\}$ and $\{b, c, 0\}$ in $\graph'$. That is, if given the color of the hyperedge relations on triples from $\set{0, a,b,c}$ which include $0$, the color of $\{a,b,c\}$ is determined. But this can be expressed as a sentence in the type of $0$, so by transitivity this must be true if we replace $0$ for any other element. That is, given four arbitrary points such that colors of three of their hyperedges form a multiset $\set{i, j, k}$, the fourth hyperedge must have color $f(\set{i, j, k})$. Define now an additional function $\hat{f}: M_n^3 \rightarrow M_n^4$ by
    $$\hat{f}(\set{i,j,k}) = \set{i, j, k, {f}(\set{i,j,k})}.$$
    We define $A$ as the image of this function. Or equivalently as the set of multisets of colors which arise as the hyperedge colors of the four sub-triples of any 4-tuple in $\graph'$. We prove that $A$ satisfies the properties in the statement:
    \begin{enumerate}
        \item Any multiset $t \in M_n^3$ is, by construction, contained in $\hat f(t) \in A$. If there was another such multiset, there would would be two 3-tuples $(a,b,c)$ and $(a', b', c')$, such that the multiset of colors of the triangle they form in $\graph$ is $t$, but the hyperedges between them in $\graph'$ have different colors, meaning they are in the same $\graph$-orbit but not $\graph'$ orbit, which is a contradiction.

        \item Suppose $\{i, i, j, j\} \not \in A$ , then $f(\{i, i, j\})$ is some $k \neq j$. That is $\set{i,i,j,k} \in A$. Consider the following subgraph of $\graph$ whose edges are colored with colors $i, j, k$:

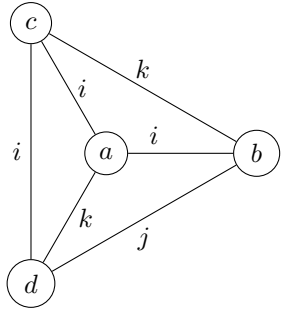
        \begin{center}
        \begin{tikzpicture}

        \node[circle, draw] (a) at (0: 0) {$a$};
        \node[circle, draw] (b) at (0: 2) {$b$};
        \node[circle, draw] (c) at (120: 2) {$c$};
        \node[circle, draw] (d) at (240: 2) {$d$};

        \draw (a) -- node[midway, above, xshift=-10pt] {$i$} (b);
        \draw (a) -- node[midway, right] {$i$} (c);
        \draw (a) -- node[midway, right] {$k$} (d);
        \draw (b) -- node[midway, above] {$k$} (c);
        \draw (b) -- node[midway, below] {$j$} (d);
        \draw (c) -- node[midway, left] {$i$} (d);
            
        \end{tikzpicture}
        \end{center}

        Notice that for every triangle, its multiset of colors is a sub-multiset of $\set{i,i,j,k}$, so $\set{a,b,d}$ and $\set{b,c, d}$ have color $i$ while $\set{a,b,c}$ and $\set{a,c, d}$ have color $j$, therefore the triples from $\{a, b, c, d\}$ have colors $\{i, i, j, j\}$, and so this multiset is in $A$, contradicting the assumption.

        \item Suppose $\set{i, i', j, j'}$ and $\set{i, i', k, k'}$ are both in $A$. Consider the following subgraph of $\graph$ whose edges are colored with colors $i, i', j, k$

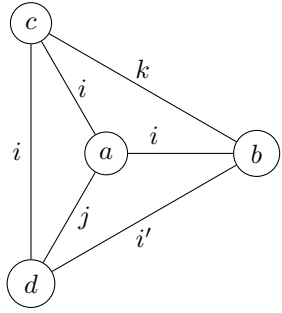
        \begin{center}
        \begin{tikzpicture}

        \node[circle, draw] (a) at (0: 0) {$a$};
        \node[circle, draw] (b) at (0: 2) {$b$};
        \node[circle, draw] (c) at (120: 2) {$c$};
        \node[circle, draw] (d) at (240: 2) {$d$};

        \draw (a) -- node[midway, above, xshift=-10pt] {$i$} (b);
        \draw (a) -- node[midway, right] {$i$} (c);
        \draw (a) -- node[midway, right] {$j$} (d);
        \draw (b) -- node[midway, above] {$k$} (c);
        \draw (b) -- node[midway, below] {$i'$} (d);
        \draw (c) -- node[midway, left] {$i$} (d);
            
        \end{tikzpicture}
        \end{center}

        Notice the following:
        \begin{itemize}
            \item The edges in $\{a,b,c\}$ are colored $\{i, i , k\} \subseteq \{i, i, k, k\} \in A$ (by (2)), so $\{a,b,c\}$ has color $k$.
            \item The edges in $\{a,b,d\}$ are colored $\{i, i', j\} \subseteq \{i, i', j, j'\} \in A$, so $\{a,b,c\}$ has color $j'$.
            \item The edges in $\{a,c,d\}$ are colored $\{i, i , j\} \subseteq \{i, i, j, j\} \in A$ (by (2)), so $\{a,b,c\}$ has color $j$.
            \item The edges in $\{b,c,d\}$ are colored $\{i, i' , k\} \subseteq \{i, i', k, k'\} \in A$, so $\{a,b,c\}$ has color $k'$.
        \end{itemize}
        Thus, the triples from $\{a, b, c, d\}$ have colors $\{j, j', k, k'\}$, and so this multiset is in $A$, as desired.
    \end{enumerate}
    Thus, $A = \textrm{Im}(\hat f)$ satisfies the conditions in the statement.
\end{proof}

We call a subset of $M_n^4$ satisfying these conditions a \emph{palette in $n$ colors}. For intuition, if $n$ is a power of 2, then we may associate to each color an element of $(\Z/2\Z)^{\log_2(n)}$. Now if we define $A$ as the 4-multisets of colors for which the sum of their colors is the identity, then it is easy to check that $A$ is a palette in $n$ colors. This is essentially the palette that the previous proof would construct if starting from the known transitive extensions of $k$-hypergraphs with $2^m$ colors. We shall show this is essentially the only way to construct palettes.

So far we have established that transitive extensions of edge-colored generic graphs with $n$ colors imply the existence of palettes in $n$ colors, but in order to obtain the full result we will require the same result for $k$-hypergraphs rather than just mere graphs. The proof of this will be broadly similar to the previous one, but with additional complexity which merits separating the results.

\begin{lemma}\label{existence of k-palettes}
    Let $k\geq 2$ and let $\graph$ be the generic edge-colored $k$-hypergraph with $n$ colors and $\graph'$ be a transitive one-point extension of $\graph$, then there exists a palette in $n$ colors.
\end{lemma}
\begin{proof}
    Use $\graph'$ to define two functions $f_k : M_n^{k+1} \rightarrow \set{1, \dots, n}$ and $\hat f_k: M_n^{k+1} \rightarrow M_n^{k+2}$ analogously to the previous proof. That is to say for a multiset of colors $t \in M_n^{k+1}$ representing the colors of all but one of the $(k+1)$-hyperedges in a set of $k+2$ points in $\graph'$, $f_k(t)$ is the color of the remaining $(k+1)$-hyperedge, and $\hat f_k(t) := t \cup \set{f(t)}$. Define $A_k$ as the image of $\hat f_k$. 
    
    We shall now construct a palette $A$ from $A_k$. Let $i_0$ be an arbitrary color, and let $A \subseteq M_n^4$ be the set of 4-multisets $t$ such that $t \cup \set{i_0^{k-2}} \in A$, where the notation $\set{i_0^{k-2}}$ refers to the multiset of $i_0$ $k-2$ times. An alternative but equivalent definition is to let $g: M_n^3 \rightarrow \set{1, \dots, n}$ be the function which maps $\{i, j, k\}$ to  $f(\set{i, j, k, i_0^{k-2}})$, define $\hat g: M_n^{3} \rightarrow  M_n^{4}$ by $\hat g(t) = t \cup \set{g(t)}$, and take $A$ equal to the image of $\hat g$.

    We will once again check that $A$ is a palette. The proof will be broadly analogous to the previous one and will even reuse the same diagrams, although they will be given a slightly different interpretation.
\begin{enumerate}
        \item Any multiset $t \in M_n^3$ is, by construction, contained in $\hat g(t) \in A$. If there was another such multiset, there would would be two $(k+1)$-sets $E, E' \subseteq G$, such that the multiset of colors of the $k$-hyperedges between them in $\graph$ is $t$, but their colors as $(k+1)$-hyperedges in $\graph'$ are different, meaning they are in the same $\graph$-orbit but not $\graph'$ orbit, which is a contradiction.

        \item Suppose $\{i, i, j, j\} \not \in A$ , then $g(\{i, i, j\})$ is some $k \neq j$. That is $\set{i,i,j,k} \in A$. Consider the following figure, reused from the previous proof:

        
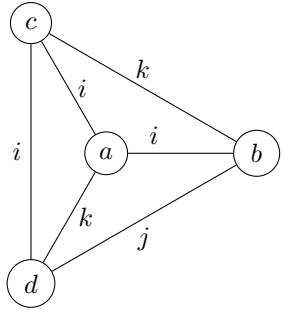
\begin{center}
        \begin{tikzpicture}

        \node[circle, draw] (a) at (0: 0) {$a$};
        \node[circle, draw] (b) at (0: 2) {$b$};
        \node[circle, draw] (c) at (120: 2) {$c$};
        \node[circle, draw] (d) at (240: 2) {$d$};

        \draw (a) -- node[midway, above, xshift=-10pt] {$i$} (b);
        \draw (a) -- node[midway, right] {$i$} (c);
        \draw (a) -- node[midway, right] {$k$} (d);
        \draw (b) -- node[midway, above] {$k$} (c);
        \draw (b) -- node[midway, below] {$j$} (d);
        \draw (c) -- node[midway, left] {$i$} (d);
            
        \end{tikzpicture}
        \end{center}

        This time, this figure does not represent a graph. Rather, Suppose there is a set $E$ of $k+2$ vertices and name four of its vertices $a, b,c,d$. If a $k$-subset of $E$ contains exactly two elements from $\set{a,b,c,d}$, color it with the color of the edge between those points in the figure. Otherwise, color it $i_0$. This is a substructure of $\graph$.

        Notice that for every $(k+1)$-subset of $E$, the multiset of colors of its $k$-subsets is a sub-multiset of $\set{i,i,j,k, i_0^{k-2}}$, To be specific: 
        \begin{itemize}
            \item $E \setminus \set{a}$ and $E \setminus \set{c}$ have $k$-hyperedges colored $\set{i,j,k, i_0^{k-2}}$, so as as $(k+1)$-hyperedges in $\graph'$ they have color $i$.
            \item $E \setminus \set{b}$ and $E \setminus \set{d}$ have $k$-hyperedges colored $\set{i,i,k, i_0^{k-2}}$, so as as $(k+1)$-hyperedges in $\graph'$ they have color $j$.
            \item The other $(k+1)$-subsets of $E$ have have $k$-hyperedges colored $\set{i,i,j,k, i_0^{k-3}}$, so as as $(k+1)$-hyperedges in $\graph'$ they have color $i_0$. There are $k-2$ of these sets.
        \end{itemize}
        Thus, the set $E$ has its $(k+1)$-hyperedges colored $\{i, i,  j, j, i_0^{k-2}\}$ and therefore this multiset is in $A_k$, and more to the point $\{i, i, j, j\} \in A$, in contradiction with the hypothesis.

        \item If one interprets the diagram the figure in the proof of property (3) in the previous lemma in the same manner as the proof of property (2) in this one, the argument becomes analogous, and has been omitted for brevity.
    \end{enumerate}
    We conclude $A = \textrm{Im}(\hat g)$ is a palette.
\end{proof}

As a side note, the set $A_k$ defined in the previous proof has properties analogous to those of a palette, and could be considered a higher-dimensional generalization of a palette. It is possible to provide an alternative proof of the following lemma by using these generalized palettes, but the arguments required become quite cumbersome compared to extracting a palette $A$ from $A_k$, as we have done.

\begin{proposition}\label{only if direction of graph theorem}
    Let $\graph$ be the generic edge-colored $k$-hypergraph with $n$ colors and suppose $\graph$ has a transitive one-point extension $\graph'$, then $n$ is a power of $2$.
\end{proposition}
\begin{proof}
    We shall prove palettes in $n$ colors only exist when $n$ is a power of 2. Suppose $A$ is a palette in $n$ colors for $n \geq 2$. We shall first prove $n$ is even.

    Take two arbitrary distinct colors $i_0, j_0$. Define a function $g$ which maps a color $k$ to the unique $\ell$ such that $\set{i_0, j_0, k, \ell} \in A$, whose existence is guaranteed by property (1). Notice that if $g(k) = \ell$ then $g(\ell) = k$, and $g$ has no fixed points since if $g(k) = k$ that implies $\{i_0, j_0, k, k\} \in A$ despite the fact that, by (2), $\{i_0, i_0, k, k\} \in A$, so $\{i_0, k, k\}$ would be contained in two elements of $A$, contradicting (1). Thus $g$ consists of disjoint transpositions which cover all of $\set{1, \dots, n}$, which can only exist if $n$ is even.

    Now, suppose $n = 2m$. We shall prove a palette in $m$ colors exists. 

    With $g$ defined as before, relabel the colors so that the first $m$ colors are all in different orbits of $g$. Define the blending operation $b: M_{2m}^4 \to M_m^4$ which replaces all instances of $i > m$ in a multiset with $g(i)$ and leaves instances of $i \leq m$ unchanged, and define a set $B = b(A) \subseteq M_m^4$. We claim $B$ is a palette in $m$ colors. Intuitively, we are identifying $i$ with $g(i)$ in all multisets of $A$ to create a new palette in half as many colors. To prove B is a palette, properties (2) and (3) are trivial, since all the sets those properties require to be in $B$, are sets which are in $A$ and feature only the colors 1 to $k$, and so they exist in $B$ as well, unchanged by $g$. The existence part of property (1) is similarly trivial; uniqueness, however, requires more work:
    
    Let $t = \{i, j, k\} \in M_m^3$. If $t' \in b\inv(t)$ and $1 \leq \ell \leq m$ is such that $t' \cup \set{\ell} \in A$ or $t' \cup \set{g(\ell)} \in A$, then $t \cup \set{\ell} \in B$. Notice that any $\ell$ with $t \cup \set{\ell} \in B$ must arise in this way, as the image $\ell$ under $g$ of a the unique element which completes a multiset in $g\inv(t)$ to an element of $A$. Otherwise, we would violate the uniqueness in property (1) of $A$, so for property (1) to hold of $B$ we must show all such $\ell$'s are in the same orbit of $g$, so that there is indeed a unique $\ell$ with $t \cup \set{\ell} \in B$.

    There are up to 8 distinct multisets in $b\inv(t) \subseteq M_n^3$, given by three binary choices between $i$ or $g(i)$, $j$ or $g(j)$ and $k$ or $g(k)$ (Note that since we do not assume $i, j, k$ to be distinct, some of these choices may be redundant, leading to a smaller $b\inv(t)$).
        
    Let us look at the particular case of two elements of $b\inv(t)$ of forms  $\{i, j, k\}$ and $\{i, j, g(k)\}$. Suppose they are completed to an element of $A$ by elements $\ell$ and $\ell'$ respectively. That is $\{i, j, k, \ell\} \in A$ and $\{i, j, g(k), \ell'\} \in A$, so by (3) we have $\{k, g(k), \ell, \ell'\} \in A$. But also by definition $\{i_0, j_0, k, g(k)\} \in A$, so by another application of (3) we obtain $\{i_0, j_0, \ell, \ell'\} \in A$, which means $\ell' = g(\ell)$. But this same logic can be applied analogously to any pair of elements in $b^{-1}(t)$ which differ in only a single entry, to show the elements which complete them are always in the same $g$-orbit. That is to say there is $1\leq\ell \leq m$ such that for any $t' \in b\inv(t)$, either $t' \cup \set{\ell} \in A$ or $t' \cup \set{g(\ell)} \in A$, and so $t$ is indeed contained in a unique element of $B$, which is $t \cup \set{\ell}$ (Note: While this argument is easiest to parse when $i, j, k$ are distinct, this is not a requirement of any step in it).

    Thus, $B$ is indeed a palette. We have proven that palettes in $n$ colors do not exist for odd $n >2$, and if they exist for an even $n$, then they exist for $n/2$. By a simple induction we may conclude palettes in $n$ colors may only exist when $n$ is a power of 2, and since the transitive extensions imply the existence of palettes, we obtain the result.

\end{proof}

Corollary \ref{if direction of graph theorem} and Proposition \ref{only if direction of graph theorem}
together prove Theorem \ref{main graph theorem}.

As a last remark, since it is known from \cite{thomas1991reducts} and the Lachlan-Woodrow classification of ultrahomogeneous graphs \cite{lachlan1980countable} that the Rado Graph is the only ultrahomogeneous graph which has proper, non-trivial reducts, from Corollary \ref{reducts corollary} we obtain
\begin{remark}
    The only ultrahomogeneous graphs which have transitive extensions are the generic graph, the empty graph and the complete graph.
\end{remark}

It is open whether this fact generalizes to higher dimensions.

\section{Tournaments and Their Generalizations}\label{tournaments}

A Tournament is a complete directed graph with no bidirectional edges. In the literature, there are conflicting ideas of how this notion should be extended to higher dimensions to define $k$-hypertournaments. One approach, going back at least to Assous \cite{assous1986enchainabilite} is to associate to each set of $k$ vertices a unique ordering of those vertices, while another, explored in some depth by Cherlin et al. \cite{cherlin2021ramsey} is to make a binary choice on each set of $k$ vertices such that the local isomorphisms on that set form precisely the alternating group $A_k$. To avoid confusion, we shall refer to the second notion as a $k$-orientation, and we will explore the transitive extensions of both $k$-hypertournaments and $k$-orientations, starting with the latter.

\subsection{k-Orientations}

\begin{definition}
    Let $k \geq 2$. A \emph{$k$-orientation} is a first order structure with a $k$-ary relation $T$ with the property that for any set of size $k$, the tuples from the set on which the relation $T$ holds are the orbit of a single (repetition-free) tuple by the action of the alternating group $A_k$.
\end{definition}

Notice that for any set of $k$ elements, half of its permutations will be in $T$. We may give a natural interpretation which motivates the name ``orientation'' by considering an embedding of $k$ points into $\mathbb{R}^{k+1}$ as linearly independent vectors $v_1, \dots, v_k$. Consider their generalized cross product, which can be thought of, mnemonically, as the ``determinant'' of a ``matrix'' whose first row is the ``vector of vectors'' composed of the unit vectors $e_1, \dots, e_{k+1}$, and whose subsequent rows are the vectors $v_i$ in some permutation. If one computes this determinant naively, it results in a vector which is perpendicular to all the $v_i$, with one of two possible orientations, and this choice of orientations is invariant under the action of $A_k$ on the permutation of the $v_i's$, but inverted under any transposition. The orientation of this cross product is then the interpretation of the $k$-orientation.

The main result of this section is the following

\maintheoremtwo*
As with the previous section, we will prove each direction as its own result.

\begin{definition}
    Let $(M, T)$ be a $(k+1)$-orientation. Two subsets $A$ and $A'$ of size $k+1$ are called \emph{near-equal} if $|A \setminus A'| = |A'\setminus A| = 1$. Given two near-equal sets $A$, $A'$, the map which takes the unique element of $A \setminus A'$ to the unique element of $A' \setminus A$ and fixes all other elements is called the \emph{match map}
    
     Two near-equal sets are said to \emph{agree} if the match map between them is \emph{not} a partial isomorphism. Two near-equal sets are said to \emph{disagree} if they do not agree.
    
    A $(k+1)$-orientation is called \emph{even} if for every set of $k+2$ elements, the number of pairs of $(k+1)$-element subsets which disagree is even.
\end{definition}

There is no intrinsic reason to state this definition for $(k+1)$-orientations rather than $k$-orientations, but we have done so to match the actual usage throughout proofs in this section. This definition of agreement might seem backwards at first glance, but the reason for choosing the nomenclature like this becomes clear by the following lemma:

\begin{lemma}
    Let $(M, T)$ be a $(k+1)$-orientation and $A \subseteq M$ be a set of size $k+2$. The agreement relation on the $(k+1)$-subsets of $A$ (with the convention that a set always agrees with itself) defines an equivalence relation. Equivalently, the disagreement relation defines a complete bipartite graph.
\end{lemma}
\begin{proof}
    It is trivial to check that the agreement relation is reflexive and symmetric. We must show it is transitive. Suppose $A_1, A_2, A_3$ are $(k+1)$-subsets of $A$ such that $A_1$ agrees with $A_2$ and $A_2$ agrees with $A_3$. We wish to show $A_1$ agrees with $A_3$. If any two of them are equal, this is trivial, so assume not. Denote by $a_{i}$ the unique element of $A$ which $A_i$ excludes, and enumerate the remaining elements common to all sets as $b_1, \dots, b_{k-1}$. We have
    $$A_1 = \set{a_2, a_3,b_1, \dots, b_{k-1}},$$
    $$A_2 = \set{a_1, a_3,b_1, \dots, b_{k-1}},$$
    $$A_3 = \set{a_2, a_1,b_1, \dots, b_{k-1}}.$$
    Without loss of generality, we may assume $\models T(a_{2}, a_{3}, \overline{b})$ (otherwise work with the complement of $T$), Note that by the agreement of $A_1$ with $A_2$, we must then have $\models T(a_{3}, a_{1}, \overline{b})$, as the alternative would render the match map  between them an isomorphism. By the same token, from the agreement of $A_2$ with $A_3$, we must have $\models T(a_1, a_2, \overline{b})$. On the other hand, if $A_1$ disagreed with $A_3$, the match map between them would be an isomorphism which yields $\models T(a_2, a_1, \overline{b})$, a single transposition away from our previous conclusion, which is a contradiction, so $A_1$ must agree with $A_3$. We conclude that agreement is, indeed, an equivalence relation.
    
    Since every size $(k+1)$ set belongs to one of two orbits in $\Aut(M, T)$, the agreement relation has two equivalence classes, which finishes the proof.
\end{proof}

We refer to the classes of the agreement equivalence relation as the \emph{agreement classes} of $A$, and to the graph of the disagreement relation as the \emph{disagreement graph}.

Now we may better understand the concept of an even orientation. For a given $A$ of size $k+2$, the number of pairs of $(k+1)$-subsets which disagree with each other is equal to the product of the sizes of the two agreement classes of $A$. When $k$ is odd, there are a total of $k+2$ subsets of size $k+1$ for a given set $A$ of size $k+2$. Since $k+2$ is therefore an odd number, one of the agreement classes will always be of even size, so all orientations are even and the notion of even orientation is trivial. For even $k$, now $k+2$ is even and the definition essentially states that the agreement classes are both of even size. Notice that if given a set of $k+2$ points, it is possible to define a $(k+1)$-orientation on it which such that the disagreement graph realizes any complete bipartite graph on $k+2$ vertices.

For further intuition, if one embeds $k+2$ points into $\R^{k+1}$ in a ``generic'' way as the vertices of a $(k+1)$-simplex, its $k$-dimensional faces are oriented by the $(k+1)$-orientation of the relevant $(k+1)$-subset. The agreement classes then correspond to the faces which are oriented towards the interior versus the exterior of the simplex. One may think of the outward-oriented faces as analogous to the set of $(k+1)$-hyperedges in a $(k+1)$-hypergraph, which makes the connection to even hypergraphs clear. Of course, Such a notion is not invariant under the choice of embedding into $\R^{k+1} $, but any two embeddings either fully agree on inwardness/outwardness of the orientations (if the linear transformation which maps one to the other has positive determinant), or fully disagree (if negative), so the notion of evenness is independent of the chosen embedding.

\begin{proposition}\label{even half of Orientation Theorem}
    Let $k \geq 2$ be an even number, and $\graph = (G, T)$ be the generic $k$-orientation. Let $\graph' = (G', T') = (G \cup \set{0}, T')$ be the unique even $(k+1)$-orientation which satisfies
    $$T(\overline{x}) \Leftrightarrow T'(\overline{x}, 0)$$
    for all $k$-tuples $\overline{x}$. Then $\graph'$ is a transitive one-point extension of $\graph$.
\end{proposition}
\begin{proof}
    We follow the spirit of the proof of Proposition \ref{Extension of Rado Graph}. Given a $(k+1)$-tuple $\overline{a}$ in $G$, we set the value of $T'(\overline{a})$ such that $\overline{a}\cup\set{0}$ satisfies the conditions of an even $(k+1)$-orientation. Explicitly, We sort the $(k+1)$-subsets of $\overline{a}\cup\set{0}$ which include 0 into disagreement classes. There are $k+1$ such subsets, which is an odd number, so exactly one of the classes must be of odd size, and we choose the $T'$-type of $\overline{a}$ so that $\overline{a}$ falls into that odd class, resulting in $\overline{a}\cup \{0\}$ having two disagreement classes of even size, as desired.

    Now we must check this satisfies the condition for being an even $(k+1)$-orientation for any set of size $k+2$ contained in $G$. We shall do this by induction. First we shall prove it for a particular $(k+2)$-subset of $G$, carefully constructed so that any two near-equal sets always agree. This property will make the proof very simple, but the actual construction takes some work.
    
    Enumerate $k+2$ elements as $G_0 = \set{a_1, \dots, a_{k+2}}$. For $1 \leq i, j\leq k+2$ let $\pi_{i j}$ denote the permutation of $G_0 \setminus \set{a_i, a_j}$ which lists its elements in order. We refer to this as the native permutation of $G_0 \setminus \set{a_i, a_j}$. Set $T(\pi_{ij})$ to be true if and only if $i+j$ is even. This defines a $k$-orientation structure $\graph_0$ which is a substructure of $\graph$, and therefore extends to a structure $\graph'_0 = (G_0\cup\set{0}, T') \leq \graph'$ in the natural way, where $T'(\pi_{ij}, 0)$ holds if and only if $i + j$ is even.

    We shall first prove any two near-equal $(k+1)$-subsets of $G_0$ which both include 0 necessarily agree. This follows essentially from the following fact: Given the permutation $(\pi_{ij}, 0)$ and a $k \neq j$, if we change the entry $a_k$ for $a_j$ in $\pi_{ij}$, the resulting permutation is in the same $A_{k+1}$-orbit as $(\pi_{ij}, 0)$ if and only if $j-k$ is odd, since to map one to the other requires $|j-k|-1$ transpositions. 
    
    Now, given two near-equal sets $A_1 = G_0' \setminus \set{a_i, a_j}$ and $A_2 = G_0' \setminus \set{a_i, a_k}$ which both include 0, the match map between them is of the form described in the previous paragraph. Therefore, the image of $A_1$'s native permutation under the match map is in the same $A_{k+1}$-orbit as $A_2$'s native permutation if and only if $j-k$ is odd, but this is also precisely the requirement for $A_1$ and $A_2$'s native permutations to be in distinct orbits of $\Aut(\graph'_0) = A_{k+1}$. Hence the match map is never a partial isomorphism, and so indeed the two sets always agree.

    What this gives us is that for any $(k+1)$-subset $A$ of $G_0$, all $(k+1)$-subsets of $A \cup \set{0}$ which include 0 agree with each other, and so $A$ must itself also agree with them all to preserve parity. At last, this lets us determine the type of a $k+1$-subset which does not include zero, with a somewhat unsurprising result: we may relabel 0 as $a_{k+3}$ and extend the notion of native permutation to subsets arbitrary subsets of $G_0'$ in the natural way, where $T(\pi_{ij})$ holds if and only $i+j$ is even, for $1 \leq i,j\leq k+3$. By repeating the prior argument, this yields the fact we wanted: Any two near-equal $(k+1)$-subsets of $\graph_0'$ indeed disagree with each other. Therefore the total number of disagreements on $(G_0',T')$ is zero, which is even, and so the particular $k$-orientation $\graph_0$ satisfies our requirements.

    Now we may proceed with the induction step. Suppose $\graph_1$ is a substructure of $\graph$ of size $k+2$ such that $G_1$ has an even number of agreements as a substructure of the extension $\graph'$. If we flip the $T$-type of a particular $k$-subset, this flips the $T'$-type of the two $(k+1)$-subsets of $G_0$ which contain it. Note that flipping the type of a $k$-subset changes it to the other disagreement class. Since we have done two such flips on $G_1$, the parities of the sizes of the disagreement classes in $G_1$ are preserved, which is what we sought.

    By induction, we may conclude that $\graph'$ is well-defined. The proof that it is a transitive extension is analogous to the proof of the same in Proposition \ref{Extension of Rado Graph}.
\end{proof}

This proves half of the main theorem. Now we prove the other half.

\begin{proposition}\label{odd half of orientation theorem}
    Let $k > 2$ be an odd number. The generic $k$-orientation $(G, T)$ has no transitive one-point extensions.
\end{proposition}
\begin{proof}
    First we must observe that if such an extension exists, then it is a $(k+1)$-orientation. By Lemma \ref{arity growth lemma} if it is in canonical form it must consist of a relation $T'$ on $G\cup\set{0}$ with $T'(\overline{x}, 0) \Leftrightarrow T(\overline{x})$. From here we may repeat the same proof as Lemma \ref{Conditions for generic Graph extensions}, but replacing arbitrary permutations with even permutations, for an analogous result

    Now, let $G_0$ be a size $k+1$ substructure of $\graph$ such that the agreement classes of the $k$-subsets of $G_0$ are of equal size, which exists since $k+1$ is even. Construct a permutation $\sigma$ of $G_0$ by taking a Hamiltonian path through the disagreement graph and mapping $a$ to $b$ if $G_0 \setminus a$ is followed by $G_0 \setminus b$ in the path. Such a path can be guaranteed to exist since the disagreement graph of $G_0$ is a complete bipartite graph with equal halves. 
    
    We claim this cyclic permutation is an automorphism of $G_0$. Indeed, for any set $A \subseteq G_0$ of size $k$, its image under $\sigma$ is, by construction, the next set on the Hamiltonian path, and therefore $\sigma(A)$ always disagrees with $A$, meaning the match map $\mu$ between them is an isomorphism. Write $A = G_0 \setminus \set{a}$ and $\sigma(A) = G_0 \setminus \set{a'}$. If $\overline{a}$ is a permutation of $A$, $\mu(\overline{a})$ simply consists of switching $a'$ for $a$, while $\sigma(\overline{a})$ consists of taking a cyclic permutation of $G_0$ which maps $a$ to $a'$ and applying it to each element of $\overline{a}$. Thus, the two permutations $\mu(\overline{a})$ and $\sigma(\overline{a})$ are a cyclic permutation (of length $k$) away from each other. Since $k$ is odd, the cyclic permutation is of even parity, and therefore both permutations have the same $T$-type. Since $\mu$ is an isomorphism, so is $\sigma$.

    Now we reach a contradiction. The permutation $\sigma$ is an automorphism of $G_0$. Therefore, $\sigma \cup \set{0 \mapsto 0}$ is an automorphism of $G_0' = G_0 \cup\set{0} \subseteq G'$, but such a thing cannot be, as that is a cycle of even length on $G_0$, which is an odd permutation and thus cannot preserve the $(k+1)$-orientation $T'$. We conclude $\graph'$ cannot exist.
\end{proof}

Propositions \ref{even half of Orientation Theorem} and \ref{odd half of orientation theorem} together prove Theorem \ref{main orient theorem}.

The fact that both the generic graph and the generic tournament have transitive extensions also yields transitive extensions for other generic binary relations. The generic directed graph with no bidirectional edges $(G, S)$ can be interpreted from the merge of the random graph $(G, R)$ and the random tournament $(G,T)$ by defining $S(a,b) \iff R(a,b) \wedge T(a,b)$. From here, it becomes a straightforward exercise to see that it has a transitive extension, given by merging the generic even 3-hypergraph $(G', R')$ and the generic even 3-orientation $(G', T')$, and interpreting a ternary relation $S'$ which is the conjunction of $T'$ and $R'$.

Similarly, the generic digraph, that is, the generic irreflexive binary relation, can be interpreted from the merge of two generic directed graphs with no bidirectional edges, by taking the disjunction of their relations, and has an analogous transitive extension.

\subsection{k-Hypertournaments}

\begin{definition}
    Let $k \geq 2$. A \emph{$k$-hypertournament} is a first order structure with a $k$-ary relation $T$ with the property that for any set of size $k$, $T$ holds on exactly one ordering of its elements. A 2-hypertournament is simply called a tournament.
\end{definition}

A useful way to think of $k$-hypertournaments is as an assignment of a linear order to the elements of any set of size $k$.

In this section we will prove generic $k$-hypertournaments have transitive one-point extensions if and only if $k = 2$. The positive result is simply the $k=2$ case of the previous section. To prove the negative result, we will first need to prove a result analogous to Lemma \ref{Conditions for generic Graph extensions}, which will become trivial in hindsight once the actual result is proven, except in the case $k=2$ where it is merely redundant, as a cyclic 3-hypertournament is isomorphic to a 3-orientation.

\begin{lemma}\label{Conditions for Tournament extensions}
    Let $\graph = (G, T)$ be a generic $k$-hypertournament. If $\graph$ has a transitive one-point extension $\graph'$ in canonical form, then $\graph'$ is a structure $(G', T')$ such that the automorphism groups induced on sets of size $k+1$ are all regular of order $k+1$ and isomorphic to each other.
\end{lemma}
\begin{proof}
      
    For any $k$-tuple $\overline{x}$ in $\graph$, the identity is the only permutation on $\overline{x}$ which is a local isomorphism (in $\graph)$, so if a permutation of $\overline{a} \cup \{0\}$ fixes $0$, it has to fix all other elements in $\overline{a}$ to be a local isomorphism (in $\graph$'), but this is part of the type of 0, so by transitivity any permutation on $k+1$ elements which fixes one of them, must fix all of them to be a local automorphism. In other words the group of local automorphisms of $\overline{x}$ is free.

    But, note that it cannot be the case that no nontrivial permutations of $k+1$ elements are automorphisms. This is because the number of $(k+1)$-orbits on $\graph'$ must coincide with the number of $k$-orbits on $\graph$, which is $k!$, and since the group of local automorphisms of $\overline{x}$ is free, by the orbit-stabilizer theorem the group of local automorphisms of $\overline{x}$ must have order $k+1$. A group of order $k+1$ may only act freely on a set of size $k+1$ if the action is regular. 

    The fact that all induced actions on sets of $k+1$ elements are isomorphic follows from the fact that $(G', T')$ is necessarily ultrahomogeneous.
\end{proof}

\begin{proposition}
    Let $k> 2$, then the generic $k$-hypertournament has no transitive one-point extension.
\end{proposition}
\begin{proof}
    Suppose for a contradiction that $\graph = (G, T)$ is a generic $k$-tournament with a transitive extension $(G, T')$. By the previous lemma there is a fixed finite permutation group $H$ of order $k+1$ which is isomorphic to the induced automorphism group of any set of size $k+1$.  

    We will deduce a contradiction by using this to obtain an extension of a generic edge-colored $k$-hypergraph with $k!$ colors, which is forbidden by Theorem \ref{main graph theorem}.

    Define an edge-colored $k$-hypergraph with $k!$ colors on $G$, where each color is indexed by an element of the symmetric group $S_k$, as follows: Fix a generic linear ordering $<$ on $G$ which is unrelated to $T$ (that is, take the merge of $(G,T)$ with a linear ordering). For every set $A \subseteq G$ of size $k$, let $T_A$ and $O_A$ be the unique orderings of $A$ given by the $k$-hypertournament $T$ and by a decreasing sequence on $<$, respectively. Color $A$ with the color of the unique permutation in $S_k$ whose natural (left) action maps $O_A$ to $T_A$. Name this $k$-hypergraph structure $(G, (R_i))$. Notice $(G, (R_i), <)$ is interdefinable with $(G, T, <)$.

    Similarly, we define an edge-colored $(k+1)$-hypergraph with $k!$ colors on $G'$ in a slightly more delicate manner: Let $\set{g_\sigma H: \sigma \in S_k}$ be the enumeration of (left) cosets of $H$ in $S_{k!}$ indexed by elements of $S_k$, where $g_\sigma$ is $\sigma \cup \set{k+1 \mapsto k+1}$. Since any two distinct permutations of a set of size $k+1$ which both end in 0 are necessarily in different cosets of $H$, and this property is part of the type of 0, this is well-defined. For every set $A' \subseteq G'$ of size $k+1$, let $T'_A$ be the restriction of $T'$ to $A'$, and let $O'_A$ be the ordering of $A$ in descending order of $<$ (with 0 taken as minimum by arbitrary convention). Color $A$ with the unique $g_\sigma$ such that $g_\sigma H \cdot \overline{a} = T'|_{A'}$. Name this edge-colored $(k+1)$-structure $(G', (R'_i))$.

    We claim that $(G,(R_i))$ is (with a change of indexes) the generic edge-colored $k$-hypergraph with $k!$ colors. Indeed, if $\mathcal{H}_0$ is a finite substructure of that hypergraph, then if we impose an arbitrary ordering on $\mathcal{H}_0$ we obtain a structure which is bi-interpretable with a finite tournament with a linear ordering by the identity map, which is a finite substructure of $(G, T, <)$, and thus embeds into it. Unwinding the chain of interpretations, this embedding turns into an embedding of $\mathcal{H}_0$ into $(G, (R_i))$, and since it is ultrahomogeneous $(G, (R_i))$ is indeed the generic edge colored $k$-hypergraph with $k!$ colors.
    
    Similarly, (with a change of indexes) $(G, (R_i'))$ is the generic edge-colored $(k+1)$-hypergraph in $k!$ colors, which follows by a similar argument.

    Notice that by construction, a permutation of $G$ is an automorphism of $(G, (R_i))$ if and only if it preserves which permutation maps $O_A$ to $T_A$ for every $k$-set $A$ in $G$, and that this happens if and only if the same permutation with $0\mapsto 0$ preserves which coset of $H$ in $S_{k+1}$ maps $O'_A$ to $T'_A$ by its action, which is equivalent to preserving the $(G',(R'_i))$-structure. Thus, we obtain that $(G',(R'_i))$ is a transitive extension of $(G, (R_i))$, despite the fact that this is impossible by Theorem \ref{main graph theorem} since such an extension can only exist when the number of colors is a power of two, which $k!$ is not for $k > 2$.

\end{proof}

\section{Equivalence Relations}

To complete the discussion of canonical homogeneous binary relations, let us examine the case of transitive extensions of generic equivalence relations. There are three meaningful cases to consider: Those with finitely many infinite classes,  those with infinitely many infinite classes, and those with infinitely many finite classes (of the same size). This last one is unique among structures in this article in that it does not have trivial algebraic closure, and hence its age does not have dAP. As such, merges with this structure are not defined, but as we shall soon see, this is not terribly relevant.

We shall prove that equivalence relations do not have transitive extensions, except in trivial cases where they are essentially pure sets (universal relation and equality relation). Our strategy will be to assume an extension exists and claim a number of ultimately contradictory properties that it must satisfy.

\begin{proposition}
    Let $\model = (M, R)$ be an equivalence relation with at least two classes, each of which is of equal cardinality of at least two. Then $\model$ has no transitive extensions.
\end{proposition}
\begin{proof}
    Suppose towards a contradiction that an extension $\model'$ in canonical form exists.

    \textbf{Claim 1:} $\model' = (M \cup \set{0}, R')$, where $R'$ is a ternary hyperedge relation which satisfies for $a,b \in M$:
    $$R'(a,b,0) \iff R(a,b).$$
    
    \textit{Proof of claim:} This is just Lemma \ref{Conditions for generic Graph extensions}, using the fact that equivalence relations are a type of graph.\hfill$\qed_{claim}$

    \textbf{Claim 2:} If $A$ is a set of size 4 and $R'$ holds on any two triples from $A$, it holds on all triples from $A$.

    \textit{Proof of claim:} By Claim 1, the relation $\sim_0$ on $M$, defined by $a \sim_0 b$ if $R'(a,b,0)$, is an equivalence relation. This much is trivial, as $\sim_0$ is just $R$, but note that this can be expressed as a first-order sentence in the type of $0$, and so must hold of all elements. That is for every $c \in M'$, the relation $a \sim_c b$ on $M' \setminus \{c\}$ defined by $R'(a,b,c)$ is also an equivalence relation. Now, suppose $\set{a,b,c,d}$ is such that $R'$ holds on two triples from it, say $\set{a,b,c}$ and $\set{a,b,d}$. Since $R'$ is symmetric under all permutations we have $R'(c,a,b) \wedge R'(a,d, b)$, and since $\sim_b$ is transitive we obtain $R'(c,d,b)$. A similar logic allows us to obtain $R'$ on $\set{a,c,d}$, which completes the proof.\hfill$\qed_{claim}$

    \textbf{Claim 3:} $R'$ satisfies for all $a,b,c \in M$
    $$R'(a,b,c) \iff R(a,b) \wedge R(b,c) \wedge R(a,c).$$

    \textit{Proof of claim:} The (unordered) type in $\model'$ of the triple $\{a,b,c\}$ must be entirely determined by the (unordered) type of the same triple in $\model$. We shall bifurcate on the latter.

    \begin{itemize}
    \item If $a,b,c$ are all in the same $R$-class, then we have $R'$ on $\{a,b,0\}, \{a,c, 0\}$ and $\{b,c,0\}$, and thus by Claim 2 we must have $R'$ on $\{a,b,c\}$. Note that this case cannot arise when the classes are of size 2.

    \item If $a$ and $b$ share an $R$-class and $c$ lies outside it, then we have $R'(a,b,0)$. To also have $R'(a,b,c)$ would imply two $R'$-relations on $\set{a,b,c,0}$, and therefore by Claim 2, all $R'$ relations on $\set{a,b,c,0}$ would hold, which is not the case, so $R'(a,b,c)$ cannot hold.

    \item If $a,b,c$ all lie in different $R$-classes, note that this case only appears when there are at least three classes, so we may assume this is the case. Recall that the setwise $4$-orbits in $\model'$ on tuples involving 0 are in bijection with the setwise $3$-orbits on $\model$, of which there are exactly three (except when classes are of size 2 in which case there are exactly two).
    
    Setwise $4$-orbits on $\model'$ are determined by the number of $R'$ relations on the quadruple, and since having the relation on two triples implies it holds on all triples, the only possible values for the number of $R'$ relations are 0, 1 and 4. In this case, four $R'$ relations on $(a,b,c,0)$ is impossible as it would imply $R$ relations that are not there, and one $R'$ relation is impossible as well since that setwise orbit is taken up by triples of the form in case 2.
    
    Thus, the only available setwise orbit is that with no $R'$ relations on $(a,b,c, 0)$, the only way this can arise is if $R'(a,b,c)$ fails to hold when $a,b,c$ are all in distinct $R$-classes.
    \end{itemize}

    \textbf{Claim 4:} $\model'$ cannot be a transitive extension of $\model$.

    \textit{Proof of claim:} Fix an $a \neq 0$. Recall we have an equivalence relation $\sim_a$ on $M' \setminus \set{a}$ defined as $b\sim_a c$ whenever $R'(b,c,a)$. If $b,c$ are in the same $R$-class as $a$ then $R'(b,c,a)$ holds by Claim 3, and $R(0, b, a)$ holds by Claim 1. So all elements of the $R$-class of $a$ (except $a$ itself) share a $\sim_a$-class with $0$. On the other hand, if $b$ is not in the same $R$-class as $a$, then by Claim 3 $R(b,c,a)$ cannot hold for any $a$, so all elements of $M$ outside the $R$-class of $a$ are in singleton $\sim_a$-classes. 
    
    This yields a contradiction, as the fact that there are singleton classes in the relation $\sim_a$ is expressible as a first-order sentence in the type of $a$, and yet by hypothesis there are no singleton classes in the relation $\sim_0$, so this sentence does not hold of $0$, meaning $0$ and $a$ have a different type and so $\model'$ cannot be transitive.
    
\end{proof}

\section{C-sets}\label{C sets}

\begin{definition}

A $C$-set is a set together with a ternary relation $C$ satisfying the following axioms.

\begin{itemize}
    \item[(C1)] $C(a; bc) \rightarrow C(a; cb)$
    \item[(C2)] $C(a; bc) \rightarrow \lnot C(b; ac)$
    \item[(C3)] $C(a; bc) \rightarrow C(a; dc) \vee C(d;bc)$
    \item[(C4)] $a \neq b \rightarrow C(a; bb)$
\end{itemize}

It is a \emph{proper} $C$-set if it in addition

\begin{itemize}
    \item[(C5)]  $\forall ab \exists c C(c; ab)$
    \item[(C5$^*$)]  $ a \neq b \rightarrow \exists c (b\neq c \wedge C(a, bc)($
\end{itemize}

And it is dense if in addition

\begin{itemize}
    \item[(C6)]  $C(a;bc) \rightarrow \exists d (C(d;bc) \wedge C(a;bd)) $
\end{itemize}
\end{definition}

A similar but distinct concept is a $D$-set:

\begin{definition}

a $D$-set is a structure together with a quaternary relation $D$ satisfying the following, similar axioms:

    \begin{itemize}
        \item[(D1)] $ \hspace{0.2in} D(wx;yz) \rightarrow \left(D(xw; yz) \wedge D(yz; wx) \right)$ 
        \item[(D2)]  $\hspace{0.2in} D(wx;yz) \rightarrow \neg D(wy;xz)$
        \item[(D3)]  $\hspace{0.2in}  D(wx;yz) \rightarrow \forall v\left(D(vx;yz) \vee D(wx;yv) \right)$
       \item[(D4)] $\hspace{0.2in} \left(w \neq y \wedge x \neq y \right) \rightarrow D(wx;yy).$
    \end{itemize}

    It is a \emph{proper} $D$-set if it contains at least three elements and in addition

    \begin{itemize}
        \item[(D5)] \hspace{0.2in} $\left(|\{w, x, y\}| = 3\right) \rightarrow \exists z \left(z \neq y \wedge D(wx;yz)\right)$
    \end{itemize}

    And it is \emph{dense} if it has at least two elements and

    \begin{itemize}
        \item[(D6)]  $\hspace{0.2in} D(wx;yz) \rightarrow \exists v \left(D(vx;yz) \wedge D(wx;yv) \right)$
    \end{itemize}
\end{definition}

The intuition for both of these concepts comes from imagining the universe of the set as the leaves of a tree. When the tree is rooted, the relation $C(a; bc)$ denotes that the unique path from a leaf node $a$ to the root is disjoint from  the unique path from $b$ to $c$. When the tree is unrooted, the relation $D(ab;cd)$ denotes that the unique path from a leaf node $a$ to $b$ is disjoint from  the unique path from $c$ to $d$. Note that while this intuition is very useful, and any finite $C$-set or $D$-set may be realized in this way, in the infinite case there exist $C$-sets and $D$-sets that do not correspond to any (graph-theoretic) tree.

The axiom sets for $C$-sets and $D$-sets are clearly analogous, as are the intuitions for them. It is known that homogeneous $D$-sets are transitive extensions of homogeneous $C$-sets, a fact which is implicitly shown in \cite[Theorem 22.1]{adeleke1998relations}. In short, given a homogeneous $C$-set on $M$, one can define a homogeneous $D$-relation on $M \cup \set{0}$ by the familiar-looking rule
$$C(x; yz) \Longleftrightarrow D(0x;yz).$$
This single rule fully determines the $D$-relation without any additional requirements. The following identity follows \cite[Theorem 23.5]{adeleke1998relations} and will be useful later, as essentially the flipside of the previous rule.
$$D(ab;cd) \iff (C(a;cd) \wedge C(b,cd)) \vee (C(c; ab) \wedge C(d, ab)).$$

In order to study $C$-relations and $D$-relations, the next concept is as crucial as the concept of an interval is to a linear order. We shall define it for $D$-sets first, as it is somewhat less cumbersome than in $C$-sets.

\begin{definition}
    A partition $\mathcal{D}$ of a $D$-set into at least three sets is called a splitting if for every $\Sigma \in \mathcal{D}$, every $a,b \in \Sigma$ and every $c, d \not \in \Sigma$, $D(ab;cd)$ holds, and for every four elements $a,b,c,d$ from distinct sets in $\mathcal{D}$, $D(ab;cd)$ does not hold.
\end{definition}
\begin{definition}
    A partition $\mathcal{C}$ of a subset of a $C$-set into at least three sets is called a splitting if there is a distinguished set $\Sigma_0 \in \mathcal{C}$ such that for any three elements $a,b,c$:
    \begin{itemize}
        \item If $a \not\in \Sigma$, $b,c  \in \Sigma$ for some $\Sigma \in \mathcal{C}\setminus\set{\Sigma_0}$, then $C(a;bc)$ holds.
        \item If $a,b,c$ come from distinct sets of $\mathcal{C}\setminus\set{\Sigma_0}$, $C(a;bc)$ does not hold.
    \end{itemize}
\end{definition}

The cardinality of a splitting is called its degree. An element of a splitting is called a sector. The distinguished sector of a $C$-set splitting is called an initial sector. 

\begin{definition}
    A $C$-set or $D$-set is called \emph{regular} if all of its splittings have the same degree.
\end{definition}

Intuitively, a splitting corresponds to removing an inner node from the tree and taking the partition of the set of leaves given by the resulting connected components, with the initial sector of a $C$-set being the connected component ``below'' the removed inner node. Thus regularity corresponds to the standard notion of regularity in trees. The following fact formalizes this intuition.

\begin{fact}
    Given two points $a,b$ in a $C$-set, there is a unique splitting $\mathcal{C}$ such that $a$ and $b$ lie in distinct non-initial sectors of $\mathcal{C}$.

    Given three points $a,b,c$ in a $D$-set, there is a unique splitting $\mathcal{D}$ such that all three lie in distinct sectors of $\mathcal{D}$.
\end{fact}

This follows for $D$-sets as a corollary of \cite[Lemma 4.2]{estrada2025model} (note this source uses a slightly different definition for $D$-set splittings. We restrict to what they call node splittings). The result for $C$-sets then follows from that using the relationship between $D$-sets and $C$-sets which we shall discuss soon.

For two points $a,b$ in a $C$-set or three points $a,b,c$ in a $D$-set, we shall call the splitting given by the previous fact their branching point.

The following fact is proven in \cite[Theorem 1]{estrada2025model} for $D$-sets only, albeit with a proof which generalizes readily to $C$-sets:

\begin{fact}
    A proper $C$-set or $D$-set is ultrahomogeneous if and only if it is dense and regular.
\end{fact}

This fact is stated only for proper $C$-sets and $D$-sets, as there are trivial non-proper examples which are ultrahomogeneous: the structures where the respective relation never holds. Such structures are essentially pure sets, and may be generally ignored.

There are three natural expansions of the language of $C$-sets which we will explore the transitive extensions of. These structures have been previously studied by Cameron \cite{cameron1987some}, largely from the perspective of counting the sizes of their orbits.

\subsection{Ordered $C$-sets}

\begin{definition}
    An ordered $C$-set is a structure $(\Theta, C, <)$, where $C$ is a $C$-relation and $<$ is a linear order, which satisfies 
    $$C(x; yz ) \rightarrow \lnot ((y<x<z) \vee ( z < x < y)).$$
\end{definition}

Cameron discusses these structures in \cite[\S 4, 5]{cameron1987some} under the name $\partial PT$. Intuitively, we may think of an ordered $C$-set as a rooted tree which has been embedded into the plane, and the ordering captures the ``left to right'' order of its leaves. This is distinct from the merge of a $C$-set with a generic linear order, which would correspond to an arbitrary ordering of the leaves with no relation to the tree structure. 

Nevertheless, the natural way to transitively extend an ordered $C$-set would be to treat it as it were a merge and extend the $C$-relation and linear ordering separately into a $D$-relation and circular ordering separately. As we shall prove, this intuition turns out to be correct, and yields a particular expansion of $D$-set by a circular order, which Cameron calls $PT$.

\begin{definition}
    A circularly ordered $D$-set is a structure $(\Omega, D, \gamma)$ where $D$ is a $D$-relation and $\gamma$ is a circular ordering which satisfies
    $$D(xy;zw) \rightarrow \lnot ([x,z,y,w] \vee [w,y,z,x]).$$
    where the bracket notation indicates all the circular relations given by that permutation of the variables under the rule ``$\gamma(x,y,z)$ if, upon repeatedly applying the permutation to $x$, it yields $y$ before $z$.
\end{definition}

The intuition here is similar to that of an ordered $C$-set, but now we may picture an embedding of an unrooted tree into the plane, and $\gamma$ is now the natural ``counter-clockwise'' circular ordering on its leaves.

\begin{proposition}
    Let $\model = (\Theta, C, <)$ be a generic ordered $C$-set, then it has a transitive extension which in canonical form is a circularly ordered $D$-set.
\end{proposition}
\begin{proof}
    Much of this result follows immediately from Lemma \ref{interpretation lemma} by interpreting both a pure $C$-set and a pure linear order from $\model$ and transitively extending both. Thus if a transitive extension of $\model$ exists, in canonical form it is $(\Theta', D, \gamma)$ for a $D$-relation $D$ and circular order $\gamma$. All we need to prove is that the resulting structure is, in fact, a circularly ordered $D$-set which is a transitive extension of $\model$, and not merely some $D$-set with some circular ordering. 

    To prove this, suppose $D(ab;cd)$. This $D$-relation may be rewritten as
    $$(C(a;cd) \wedge C(b,cd)) \vee (C(c; ab) \wedge C(d, ab)).$$
    Each of these $C$-relations imply a large disjunction of possible $<$-types for $\set{a,b,c,d}$. Crucially, $C(a;cd) \wedge C(b,cd)$ exclude any $<$-types where $a$ and $b$ are between $c$ and $d$, and only those $<$-types induce the circular orderings $[a,c,b,d]$ and $[d,b,c,a]$, so $D(ab;cd)$ does indeed exclude those orderings, as desired. For the disjunct $(C(c; ab) \wedge C(d, ab))$ the logic is analogous. We conclude that this defines a circularly ordered $D$-set, and to check that it is a transitive extension of $\model$ is straightforward.
    
\end{proof}

\subsection{Colored $C$-sets}\hfill\\

There are two natural notions of a colored $C$-set, the first is simply a generic coloring on the elements, which is just the merge of a $C$-set with a colored pure set, which is not even transitive and thus will not be examined. The second is somewhat more interesting:

\begin{definition}
    An \emph{internally-colored $C$-set} is a $C$-set $\Theta$ augmented with binary relations $P_1, \dots P_n$ called colors which partition $\Theta^2 \setminus \Delta_\Theta^2$, such that for any splitting $\mathcal{C}$ of $\Theta$, there is a unique color $P_i$ such that $P_i(x,y)$ for all $x,y$ whose branching point is $\mathcal{C}$
\end{definition}

 Intuitively, what this definition seeks to capture is the notion of applying a coloring on the inner nodes of a rooted tree, formally understood as branching points of pairs. Thus, if two pairs of leaves have the same branching point, they ought to have the same color. This expansion was studied by Cameron \cite[\S 5]{cameron1987some} under the name $\partial T(r)$ where $r$ is the number of colors. There is a natural extension of this notion to $D$-sets, which he calls $T(r)$:

\begin{definition}
    An \emph{internally-colored $D$-set} is a $D$-set $\Omega$ augmented with ternary relations $P_1, \dots P_n$ called colors which partition $\Omega^3 \setminus \Delta_\Omega^3$, such that for any splitting $\mathcal{D}$ of $\Omega$, there is a unique color $P_i$ such that $P_i(x,y, z)$ for all $x,y,z$ whose branching point is $\mathcal{D}$.
\end{definition}

Again, intuitively this colors the inner nodes of an unrooted tree, where each inner node is formally identified as a branching point of triples of leaves. 

\begin{proposition}
    Let $\model = (\Theta, D, P_1, \dots, P_n)$ be a generic internally-colored $C$-set, then $\model$ has a transitive extension which in canonical form is an internally-colored $D$-set.
\end{proposition}
\begin{proof}
    Similar to the previous subsection, using Lemma \ref{interpretation lemma}, the fact that the extension in canonical form must consist of a $D$-relation $D$ and colors $P_i'$ follows immediately from the fact that $(\Theta, C)$ is a $C$-set and that $(\Theta, P_1, \dots, P_n)$ is an edge-colored graph, the latter of which follows since colors are symmetric from their construction. As before, all we really have to check is that this defines an internally-colored $D$-set which extends $\model$, rather than some $D$-set $\Omega$ with some partition of $(\Omega^3\setminus \Delta_{\Omega}^3)$ into $n$ sets.

    To check this, we use the natural bijection between splittings of a pure $C$-set and those of its extension to a $D$-set. Namely, to each splitting $\mathcal{C}$ of the $C$-set with initial sector $\Sigma_0$, we associate the splitting of the $D$-set given by $\mathcal{C}\cup \set{\Sigma_0 \cup \set{0}} \setminus\set{\Sigma_0}$. That is the splitting of the $D$-set obtained by adding 0 to the initial sector. To check that this yields a bijection between the splittings of $\Theta$ and the splittings of its extension is merely to check the definitions. It then follows quickly from this construction that this bijection in fact preserves the colors of splittings, which yields that $\model'$ is indeed an internally-colored $D$-set which transitively extends $\model$.
\end{proof}

Note that we did not invoke any sort of parity for the colored 3-hypergraph $(\Omega, P_1, \dots, P_n)$, but a form of evenness actually follows, since for any set $A$ of four points in a $D$-set, either there is one common branching point for all triples from $A$ (this is the case if no non-trivial $D$-relations hold on $A$, that is when the induced unrooted tree is the one with one inner node and four leaves), or there are two splittings $\mathcal{D}_1, \mathcal{D}_2$ such that one is the branching point of two triples and the other is the branching point of the other two (This is the case where some $D$-relation holds, that is when the induced unrooted tree is the one with two inner nodes and four leaves). In the first case, all four triples necessarily have the same color, and in the second, one pair will have one color and the other will have another (or possibly the same) color. In any case, the result is always that each color forms an even 3-hypergraph.

If we forget the $C$-set structure, this result also provides a new example of a class of graphs with transitive extensions which are not generic. Namely, this establishes transitive extensions for edge-colored graphs which can be realized as internal colorings of $C$-sets (even those whose number of colors is not a power of 2). We finalize this section with a remark on such graphs, which is well-known in the literature as the correspondence between cographs and cotrees.

\begin{definition}
    A size 4 graph $(G, R)$ is called an \emph{N} if there is an enumeration $\set{a_1, a_2, a_3, a_4}$ of its vertices such that its set of edges is $\set{\set{a_1,a_2}, \set{a_2,a_3}, \set{a_3,a_4}}$
    
    An edge-colored graph is called \emph{N-free} if for every set $A$ of four vertices and color $R_i$, $(A, R_i)$ is not an N.
\end{definition}

\begin{proposition}
    Let $\model = (\Theta, C, R_1, \dots, R_n)$ be an internally-colored $C$-set, then $\graph = (\Theta, R_1, \dots, R_n)$ is $N$-free.
\end{proposition}
\begin{proof}
    Let $A$ be a subset of $\Theta$ of size 4. If there is a triangle or anti-triangle in $(A, R_i)$ then it is not an $N$ and we are done, so suppose there are none. Thus, it cannot be that there is some splitting such that three elements of $A$ are in three different sectors, as all three pairs from those three elements would have the same branching point and thus the same color, yielding a triangle or anti-triangle. That is to say, we may suppose $(A, C)$ is the $C$-set of a binary tree. Up to isomorphism, there are two binary trees with four nodes, pictured below:
    
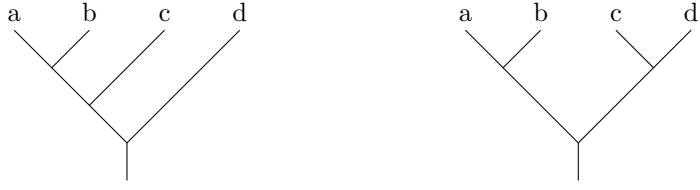
\begin{center}
    \begin{tikzpicture}

\begin{scope}[xshift=0cm]
  \draw (0,1.5) -- (-1.5,3) node[above] {a};
  \draw (0,1) -- (0,1.5);

  \draw (-1, 2.5) -- (-0.5,3) node[above] {b};
  \draw (-0.5,2) -- (0.5,3) node[above] {c};
  \draw (0,1.5) -- (1.5,3) node[above] {d};
\end{scope}

\begin{scope}[xshift=6cm]
   \draw (0,1) -- (0,1.5);

  \draw (0, 1.5) -- (-1.5,3) node[above] {a};
  \draw (-1, 2.5) -- (-0.5,3) node[above] {b};

    \draw (1,2.5) -- (0.5,3) node[above] {c};
  \draw (0, 1.5) -- (1.5,3) node[above] {d};

\end{scope}

\end{tikzpicture}
\end{center}
A minimum requirement for an N is that exactly three edges are colored. In the right hand tree, this is impossible, as the bottom-most node is the branching point of four pairs of points, while each of the upper nodes is the branching point of only one. Thus, no matter how the branching points are colored, there will never be three colored edges. The left hand tree is more promising; from bottom to top, the inner nodes are branching points of three, two and one pair, respectively, so three colored nodes are possible if only the bottom-most one is colored or if both of the upper ones are. The first option yields an anti-triangle on $\set{a,b,c}$ while the first yields a triangle on the same set. Thus, $(A, R_i)$ cannot be an N and we are done.
\end{proof}

Of note, the class of N-free graphs does not have the amalgamation property, and thus there is no generic N-free graph. Nevertheless, N-free graphs which arise as reducts of internally-colored $C$-sets have transitive extensions, which are the reducts of the corresponding internally-colored $D$-sets. More detail on $N$-free graphs and how their lack of an amalgamation property can be salvaged can be found in \cite[\S 5]{cameron1987some} and \cite{covington1989universal}.

\subsection{Leveled $C$-sets}

\begin{definition}
    A \emph{leveled $C$-set} is a $C$-set augmented with a quaternary relation $L(x_1x_2;y_1y_2)$ called a leveling, symmetric under the permutations (12) and (34), which we abbreviate as $x_1x_2 \preceq y_1y_2$. which forms a linear preorder on (unordered) pairs of elements, and respects 

\begin{equation}
    C(a;bc) \leftrightarrow (L(ab; bc) \wedge \lnot L(bc;ab)).
\end{equation}
\end{definition}

We may shorthand the latter statement as $ab \prec bc$. Leveled $C$-sets are a natural way to expand a $C$-set by interpreting a ``height'' linear preorder on the inner nodes of a tree, with each inner node represented by the pairs of points have it as branching point. Cameron \cite[\S 6]{cameron1987some} studied this structure under the name $\partial T(\prec)$ and further study on it was done recently by Almazaydeh, Braunfeld and Macpherson \cite{almazaydeh2024omega} under the name $M_{C, lev}$, where they observed that they can encode permutations, a rather unusual feature among expansions of $C$-sets which precludes them from being monadically NIP, unlike ordered and colored $C$-sets.

One could expect to be able to define an analogous expansion of a $D$-set which extends a leveled $C$-set, yielding something that in Cameron's notation would no doubt be termed $T(\prec)$, but we will show that this is not the case, showing another way in which this expansion has a different flavor to others discussed in this section, and going some ways towards explaining why Cameron could not define $T(\prec)$.

\begin{definition}
    Let $I$ be a linear order.
    
    A sequence $(a_i)_{i \in I}$ in a $C$-set is called $C$-monotonic if $C(a_i; a_ja_k)$ for any $i < j < k$. 
    
    A sequence $(a_i)_{i \in I}$ in a $D$-set is called $D$-monotonic if $D(a_ia_j;a_ka_\ell)$ for all $i < j < k < \ell$.
\end{definition}

\begin{proposition}
    An ultrahomogeneous leveled $C$ does not have a transitive extension.
\end{proposition}
\begin{proof}
    Let $\model = (\Theta, C, L)$ be an ultrahomogeneous leveled $C$-set. We begin by noting that any transitive extension $\model'$ of $\model$ must have a reduct that is isomorphic to a $D$-set. This is essentially a case of Lemma \ref{interpretation lemma}: if we name the constant $0$ then $(\model', 0)$ has a reduct which is (if we take the constant out of the domain) isomorphic to $(\Theta, C)$. That is there is a reduct of $(\model',0)$ which includes the constant $0$ and is bi-interpretable with a $C$-set, but $C$-sets are themselves bi-interpretable with a $D$-set with a named constant, so $\model'$ must interpret a $D$-relation. Following the bi-interpretations this $D$-relation is precisely the canonical $D$-set associated to $(\Theta, C)$. 

    Now, notice that in $(\Theta, C, L)$, any $C$-monotonic sequence of $n$ elements is isomorphic to any other such sequence, as the two will have the same $C$-relation structure which fully determines the $L$-structure on the sequence as each branching point must necessarily be higher than the previous ones.

    In particular, if there are elements $b,b',c,d,d'$ and $e$ such that $(b, c, d, e)$ and $(b', c, d',e)$ are both $C$-monotonic, then there is an isomorphism mapping one to the other. This has the consequence that in $\model'$, if $(0, b, c, d, e)$ and $(0, b', c, d', e)$ are both $D$-monotonic, there must be an isomorphism mapping one to the other, but being the first element of a $D$-monotonic sequence of 5 elements is a definable condition, and so this property of $0$ is contained in the type of $0$.

    But by transitivity the type of $0$ is the same as the type of any other element, so if $a,b,b', c, d,d',e$ are such that $(a,b,c,d,e)$ and $(a, b', c, d', e)$ are both $D$-monotonic, then they have the same type. However this yields a contradiction, as it is possible to pick $c = 0$ and then pick the other elements such that $L(ab'; d'e)$ but not $L(ab, de)$, as illustrated below, where $(a,b,b',0,d,d',e)$ forms a $D$-monotonic sequence which has been laid out in a way that emphasizes the original $C$-set structure, and the dashed lines represent the classes of the preorder $L$.
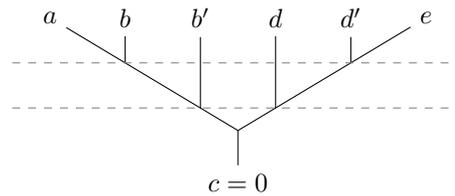
    \begin{center}
    \begin{tikzpicture}
        \node (c) at (0, 1.3) {$c = 0$};
        \node (a) at (-2.5, 3.5) {$a$};
        \node (b) at (-1.5, 3.5) {$b$};
        \node (bb) at (-0.5, 3.5) {$b'$};
        \node (d) at (0.5, 3.5) {$d$};
        \node (dd) at (1.5, 3.5) {$d'$};
        \node (e) at (2.5, 3.5) {$e$};
        \draw (c) -- (0, 2);
        \draw (a) -- (0, 2);
        \draw (b) -- (-1.5, 2.9);
        \draw (bb) -- (-0.5, 2.3);
        \draw (e) -- (0, 2);
        \draw (dd) -- (1.5, 2.9);
        \draw (d) -- (0.5, 2.3);
        \draw[dashed, gray] (-3, 2.9) -- (3, 2.9);
        \draw[dashed, gray] (-3, 2.3) -- (3, 2.3);
    \end{tikzpicture}
    \end{center}

    In other words, the existence of a transitive extension demands the existence of an automorphism fixing $a, 0, e$ and mapping $b \mapsto b'$, $d \mapsto d'$, but such an automorphism fixes $0$ and therefore must be an automorphism in $\model$, but no such automorphism exists as such a map would break the $L$-structure.
\end{proof}

\bibliographystyle{babplain}
\bibliography{ref}

\end{document}